\documentclass[a4paper,10pt]{article}%
\usepackage{amsmath}%
\usepackage{amsfonts}%
\usepackage{amssymb}%
\usepackage{amsthm}
\usepackage{graphicx}
\usepackage{xcolor}
\usepackage[utf8]{inputenc}
\usepackage[T1]{fontenc}
\usepackage[english,french]{babel}

\newtheorem{theorem}{Theorem}[section]

\newtheorem{corollary}[theorem]{Corollary}

\newtheorem{definition}[theorem]{Definition}

\newtheorem{lemma}[theorem]{Lemma}

\newtheorem{proposition}[theorem]{Proposition}
\newtheorem{remark}[theorem]{Remark}

\newcommand{\R}{\mathbb{R}}

\newcommand{\N}{\mathbb{N}}

\newcommand{\E}{\mathbb{E}}

\newcommand{\ep}{\varepsilon}

\newlength{\espaceavantspecialthm}
\newlength{\espaceapresspecialthm}
{\\}

\newenvironment{defi*}[1][]{
\vskip \espaceavantspecialthm \noindent \textbf{D\'efinition.} }%
{\vskip \espaceapresspecialthm}

\begin{document}

\sloppy

\title{Random actions of homeomorphisms of Cantor sets embedded in a line and Tits alternative}
\author{Dominique Malicet, Emmanuel Militon}
\maketitle

\setlength{\parskip}{10pt}

\begin{abstract}
In 2000, Margulis proved that any  group of homeomorphisms of the circle either preserves a probabilty measure on the circle or contains a free subgroup on two generators, which is reminiscent of the Tits alternative for linear groups. In this article, we prove an analogous statement for groups of locally monotonic homeomorphisms of a compact subset of $\mathbb{R}$. The proof relies on dynamic properties of random walks on the group, which may be of independent interest.
\end{abstract}

\selectlanguage{english}

\section{Introduction}
\subsection{Context}
In \cite{Tits}, Tits proves that any subgroup of the linear group of $\R^n$ either is solvable or contains a free subgroup of rank $2$. It is called the Tits alternative. Since then, analogous statements have been found in many other contexts, and one usually calls a Tits alternative any statement telling that every subgroup of a given group either contains a free group of rank $2$, or has some kind of rigidity. For example let us state Margulis Theorem about subgroups of the group $\mathrm{Homeo}(\mathbb{S}^1)$ of homeomorphisms of the circle.

\begin{theorem}  \label{Margulis} \cite{Mar}\\
 Let $G$ be a subgroup of $\mathrm{Homeo}(\mathbb{S}^1)$. Then:
	\begin{enumerate}
		\item either the group $G$ contains a nonabelian free subgroup,
		\item or the action of $G$ on $\mathbb{S}^1$ preserves a probability measure.
	\end{enumerate}
\end{theorem}

Other proofs of this theorem exist, for example in \cite{Maly}, \cite{Ghys}, \cite{Nav}. For other analogous statements in other contexts, let us cite for example \cite{A}, \cite{BG}, \cite{HM} and \cite{MC}.

The proofs of Theorem \ref{Margulis} always rely on the fact that homeomorphisms of the circles preserve or reverse the circular order. This rigidity allows to find some structure on the group actions. On the other hand, it is very hard to obtain some kind of Tits alternative for groups of homeomorphisms of a higher dimensional manifold (see however \cite{HX} for a result in this direction). In this article, we study an intermediate case: if $K\subset \mathbb{S}^1$ is a compact subset of the circle, we consider groups of homeomorphisms of $K$ which locally preserve or reverse the order. If $K= \mathbb{S}^1$, the group is simply $\mathrm{Homeo}(\mathbb{S}^1)$, but if $K$ is smaller, typically a Cantor set, then the structure of the group actions may complexify. Indeed, some  elements of the group might permute various parts of $K$. In fact, if $K$ is a Cantor set it is possible to prove that any finite group acts faithfully on $K$ in a locally monotonic way.

We prove a Tits alternative for this kind of actions. The usual proofs of Margulis Theorem do not adapt well to this more general context, and we actually did not manage to find a purely topological proof: our proof relies on probabilistic tools, and more precisely on properties of random walks on the group.

\subsection{Main theorem}

Let us give a precise statement of our theorem. We will actually consider compact subsets of $\R$ instead of $\mathbb S^1$ to simplify a bit the definition.  Let us begin by defining the notion of local monotonicity.

\begin{definition}\label{locmon}
	 If $E\subset \R$, a map $f:E\rightarrow \R$ is said to be locally monotonic if it is monotonic on some neighborhood of every point of $E$ (endowed with the order induced by $\R$). Equivalently, for any $x$ in $E$ there exists an open interval $I$ containing $x$ such that $f|_{I\cap E}$ is monotonic.
\end{definition}

Our main result is the following.

\begin{theorem} \label{Tits} 
	Let $K$ be a compact subset of $\R$. Let $G$ be a subgroup of $\mathrm{Homeo}(K)$ whose elements are locally monotonic on $K$. Then:
	\begin{enumerate}
		\item either the group $G$ contains a nonabelian free subgroup,
		\item or the action of $G$ on $K$ preserves a probability measure.
	\end{enumerate}
\end{theorem}

\begin{remark}
Our proof would adapt very well if $K$ was supposed to be a subset of $\mathbb S^1$ instead of a subset of $\R$. However, in the case $K\not=\mathbb S^1$, we can suppose $K \subset \mathbb{R}$ by cutting the circle at a point outside $K$, and the case $K=\mathbb S^1$ is a direct application of Margulis Theorem. Hence such a statement would not be more general.
\end{remark}

\subsection{Applications}

\subsubsection*{Groups of diffeomorphisms of a Cantor set}

We can naturally apply Theorem \ref{Tits} to groups of diffeomorphisms of a Cantor subset of the real line. This kind of group has been studied for exemple in \cite{FN} and \cite{MM}. We refer to these articles for precise definitions.  Let us simply recall here that a metric space $K$ is a Cantor set if it is infinite, any of its points is an accumulation point and it is totally disconnected, and that a diffeomorphism of a Cantor set of $\R$ is a homeomorphism which is locally the restriction of a diffeomorphism. In particular it is locally monotonic. Finally we recall that if $K$ is the standard ternary Cantor set then its group of diffeomorphisms is Thompson's group $V$.\\

The following result is an immediate corollary of Theorem \ref{Tits}.

\begin{corollary}
	Let $K\subset \R$ be a Cantor set and $G$ be a subgroup of $\mathrm{Diff}^1(K)$, then:
	\begin{enumerate}
		\item either the group $G$ contains a nonabelian free subgroup,
		\item or the action of $G$ on $K$ preserves a probability measure.
	\end{enumerate}
\end{corollary}

If we assume that the diffeomorphisms are $C^2$ and that the group is finitely generated, then we can use Sacksteder theorem to deduce that in the second case of the alternative,  the group actually has a finite orbit. The following corollary answers a question by Funar and Neretin (see the last section of \cite{FN}) who asked whether any subgroup of the group of diffeomorphisms of a Cantor subset of the real line either had a finite orbit or contained a free subgroup on two generators. Observe that, when you replace $\mathrm{Diff}^2(K)$ by $\mathrm{Diff}^1(K)$, the corresponding statement does not hold anymore because of the existence of Denjoy counterexamples : the group generated by a Denjoy counterexample is a subgroup of the group of diffeomorphisms of its minimal subset and does not have any finite orbit.

\begin{corollary}
If $K\subset \R$ is a Cantor set, if $G$ is a finitely generated subgroup of $\mathrm{Diff}^2(K)$, then one of the two following properties holds.
	\begin{enumerate}
	\item either the group $G$ contains a free subgroup of rank $2$,
	\item or the group $G$ has at least one finite orbit.
\end{enumerate}
\end{corollary}


\begin{proof}
	If $G$ does not contain a nonabelian free subgroup, then, by Theorem \ref{Tits}, there exists a probability measure $\mu$ on $K$ which is $G$-invariant. Then the support $\mbox{supp}(\mu)$ of $\mu$ is $G$-invariant. Hence it contains a minimal $G$-invariant closed subset $F$, which is either a finite set or a Cantor set. In the first case we are done. In the second case we use Sacksteder theorem to say that there exists $x$ in $F$ which is a hyperbolic fixed point of an element $g$ of $G$, that we can assume attracting. Then if $U$ is the basin of attraction of this fixed point, the invariance of $\mu$ by $g$ implies that the restriction of $\mu$ to $U$ is a multiple of $\delta_x$. So $x$ is isolated in $F$, which is a contradiction since the subset $F$ is a Cantor set.
\end{proof}

%

\subsubsection*{Generalized interval exchange transformations}

In what follows, we give a consequence of Theorem \ref{Tits} for groups of generalized interval exchange transformations. Let us recall the definitions of interval exchange transformations and generalized interval exchange transformations.

\begin{definition}
	Let $I=[a,b]$ be a compact interval of $\R$.
	
	A map $f:I\rightarrow I$ is an interval exchange transformation (IET) of $I$ if it is one to one, piecewise continuous, and its restriction to every interval on which it is continuous is a translation.
	
	A map $f:I\rightarrow I$ is a generalized interval exchange transformation (GIET) of $I$ if it is one to one and piecewise continuous.
\end{definition}

	
	

Composition gives a group structure to the sets of interval exchange transformations and of generalized interval exchange transformations of an interval $I$. Those groups are respectively denoted by $IET(I)$ and $GIET(I)$. 
We are going to obtain a Tits alternative for subgroups of $GIET(I)$. However, we need to be careful with the notion of invariant measure because of the discontinuity points. We will use the following definition.

\begin{definition}\label{GIETmeasure}
We say that a subgroup $G$ of $GIET(I)$ preserves a probability measure on $I$ if either there exists a $G$-invariant probability measure on $I$ without atoms or there exists a point $x\in I$ whose left orbit $O^-(x)=\{g(x^-),g\in G\}$ or right orbit $O^+(x)=\{g(x^+),g\in G\}$ is finite, where $g(x^-)$ (resp. $g(x^+)$) is the left (resp. right) limit of $g$ at $x$.
\end{definition}

With this definition we can formulate a Tits alternative in the case of generalized interval exchange transformations.

\begin{theorem}\label{TitsIET}
	Let $I=[a,b]$ be a compact interval of $\R$. Let $G$ be a finitely generated subgroup of $GIET(I)$. Then one of the following properties holds.
	\begin{enumerate}
		\item The group $G$ contains a free subgroup of rank $2$;
		\item The group $G$ preserves a probability measure on $I$.
	\end{enumerate}	
\end{theorem}

In the second case, either $G$ has a left orbit $O^-(x)=\{g^-(x),g\in G\}$ or a right orbit $O^+(x)=\{g^+(x),g\in G\}$ which is finite, or there exists an invariant measure without atoms, and then its repartition function defines a semiconjugacy from $G$ to a subgroup $H$ of $GIET(I)$ which preserves the Lebesgue measure, so actually $H$ is a group of interval exchange transformations. We formalize this remark with the following statement.

\begin{corollary}
Let $I=[a,b]$ a compact interval of $\R$. Let $G$ be a finitely generated subgroup of $GIET(I)$. Then one of the three following properties holds.
\begin{enumerate}
		\item The group $G$ contains a free subgroup of rank $2$;
		\item The group $G$ has either a left orbit $O^-(x)=\{g^-(x),g\in G\}$ or a right orbit $O^+(x)=\{g^+(x),g\in G\}$ which is finite.
		\item There exist an onto continuous nondecreasing map $h : I \rightarrow [0,1]$, a subgroup $H$ of $IET([0,1])$ and a morphism $\varphi:G \rightarrow H$ such that, for any element $g \in G$,
		$$ hg=\varphi(g)h.$$
	\end{enumerate}	
\end{corollary}

\begin{proof}[Proof of Theorem \ref{TitsIET}]
	Let us denote by $D$ the union of the discontinuity points of all the elements of $G$. The set $D$ is countable, and $G$ acts continuously on $I\backslash D$. We are going to blow up all the points of $D$ in non trivial intervals, so that $I\backslash D$ becomes a closed set on which $G$ acts continuously. To do it rigourously, we take a continuous increasing function $f:I=[a,b]\to J=[c,d]$ whose set of discontinuity points is exactly $D$. Such a function exists because $D$ is countable (take for exemple the repartition function of $\nu=\lambda+\sum_{c\in D} \alpha_c \delta_c$ where $\lambda$ is the Lebesgue measure and $(\alpha_c)_{c\in D}$ a summable family of positive real numbers). Then $f$ is an increasing homeomorphism of $I\backslash D$ onto $Y=f(I\backslash D)$. Note that the boundary of $Y$ are the points $f(c^-),f(c^+)$ for $c$ in $D$, so in particular they are either isolated from the left or from the right. Then, for any $g$ in $G$, $fgf^{-1}$ is a locally monotonic homeomorphism of $Y$ which can be continuously extended to a homeomomorphism $\hat{g}$ of $K=\overline{Y}$.\\
	
	Now we apply Theorem \ref{Tits} to the action on $K$ of the group $\hat{G}=\{\hat{g},g\in G\}$. Three cases can occur.
	
	\begin{enumerate}
		\item If $\hat{G}$ contains a non trivial free group, so does $G$.
		\item If $\hat{G}$ preserves a probability measure $\mu$ on $K$ without atoms, then $\mu$ is actually a probabilty measure on $Y$, and $G$ preserves the probabilty measure $\nu=f^{-1}_*\mu$.
		\item If $\hat{G}$ preserves a probability measure $\mu$ on $K$ with at least an atom at a point $p$, then $\{x\in K| \mu(\{x\})=\mu(\{p\})\}$ is a nonempty finite subset which is invariant under the action of $\hat{G}$, so the orbit of $x$ under $\hat{G}$ is finite. If $x$ is of the form $f(c^-)$ for some $c$ in $D$, then $G$ has a finite left orbit, if it is of the form $f(c^+)$ then  $G$ has a finite right orbit, and if not then $G$ has a finite (standard) orbit. In any case $G$ preserves a probability measure in the sense of Definition \ref{GIETmeasure}.
	\end{enumerate}
	
\end{proof}

\begin{remark}\label{lexico}
If $I=[a,b]$,  we can considered $\tilde{I}$ the signed interval $\tilde{I}=I\times \{-1,1\}$ minus the two extreme values $(a,-1)$ and $(b,1)$. It can be endowed with the topology induced by the lexicographic order on $\tilde{I}$. Then any element $f\in GIET(I)$ acts continuously on $\tilde{I}$ by setting $\tilde{f}(x,1)=(f(x^+),1)$ and $\tilde{f}(x,-1)=(f(x^-),-1)$. Then $GIET(I)$  may be identified as a subgroup of $\mbox{Homeo}(\tilde{I})$, and then the notion of invariant measure for subgroups of $GIET(I)$ corresponds to the classical notion for subgroups of $\mbox{Homeo}(\tilde{I})$. 
\end{remark}

\subsubsection*{Groups action on linearly ordered set}
Theorem \ref{Tits} can be expressed in a more abstract setting of group actions on linearly ordered sets. Let $E$ be a set endowed with a linear order $\leq$ (i.e. every two elements of $E$ are comparable) and the corresponding order topology. We will assume that $E$ is complete as a lattice, which means that every nonempty subset of $E$ has a supremum and an infimum, and is equivalent to the fact that $E$ is compact as a topological space (see \cite{B}, Chapitre 3, Theorem 9 for a proof). The notion of local monotonic transformation of $E$ can be defined as in Definition \ref{locmon}. Then we can state a Tits alternative.

\begin{theorem} \label{TitsLinOrd}
Let $(E,\leq)$ a linearly ordered set which is complete as a lattice. Let $G$ be a group of locally monotonic homeomorphisms of $E$. Then:
	\begin{enumerate}
		\item either the group $G$ contains a nonabelian free group,
		\item or the action of $G$ on $E$ preserves a probability measure.
	\end{enumerate}
\end{theorem}
Theorem \ref{Tits} naturally enters in this context, but also Theorem \ref{TitsIET} thanks to Remark \ref{lexico}.

Though $E$ and so $G$ may have a wild structure (e.g. every successor ordinal is a complete lattice), it is possible to significantly simplify the problem in order to come back to the context of Theorem \ref{Tits}. To do this, we will use the following general lemma.
\begin{lemma}\label{reduction}
	Let $X$ be a topological compact space and let $G$ be a subgroup of $\mathrm{Homeo}(X)$. If every finitely generated subgroup of $G$ preserves a probability measure on $X$, then so does $G$.
\end{lemma}

\begin{proof}
	Let $\Pi$ be the set of probability measures on $X$ endowed with the weak-$*$ topology. For any $g$ in $G$, let $\Pi(g)=\{\mu\in \Pi| g_*\mu=\mu\}$. This is a closed subset of $\Pi$. For any $g_1,\ldots,g_n$ in $G$, by assumption on $H=\langle g_1,\ldots, g_n\rangle$, we have $$\bigcap_{k=1}^n \Pi(g_k)\not=\emptyset.$$
	As the space $\Pi$ is compact, we deduce that $$\bigcap_{g\in G} \Pi(g)\not=\emptyset.$$
\end{proof}
Let us see how Theorem \ref{TitsLinOrd} is implied by Theorem \ref{Tits}.
\begin{proof}
The assumptions imply that $E$ is compact as a topological space. Let $G\subset \mathrm{Homeo}(E)$ as in the statement. Thanks to Lemma \ref{reduction}, we can assume that $G$ is finitely generated, and in particular, countable. 

Then, by restricting the action of $G$ on a closed orbit of some point, we can assume that $E$ is separable. 

Next we call gap (or jump) of $E$ a pair of elements $J=\{x,y\}$ which are consecutive for the order, i.e. $x\not=y$ and there does not exists a point of $E$ between $x$ and $y$. For $g$ in $G$ fixed, using that $g$ and $g^{-1}$ are locally mononotonic, every point of $E$ admits a neighborhood $V$ such that for any gap $J\subset V$, its image $g(J)$ is still a gap. By compactness we deduce that the image by $g$ of every gap is a gap except for a finite number of exception. Since $G$ is countable, we deduce that every gap $J$ except a countable number of exceptions satisfies that its image by every element of $G$ is still a gap. By identifying the two points of each of these gaps, we obtain a new ordered set where $G$ still acts, satisfying the same assumptions, and having only a countable number of gaps. So we can assume that $E$ has a countable number of gaps.

Now, by \cite{F}, a linearly ordered set which is separable and has a countable number of gaps is isomorphic to a subset of $\R$. So we can assume $E\subset \R$, and since $E$ is compact, Theorem \ref{Tits} applies.
\end{proof}

\subsection{The Ping Pong argument}
The proof of the original Tits alternative, as well as Margulis version, rely on the well-known ping-pong lemma. We recall a version of this lemma (see \cite{dlH} p.25 for a proof).

\begin{lemma}[Ping-pong lemma]
	Let $G$ be a group acting on a set $X$ and let $a_1,a_2$ be elements of $G$. Suppose there exist pairwise disjoint subsets $A_1$, $B_1$, $A_2$ and $B_2$ such that $a_i(X\backslash A_i)\subset B_i$ for $i=1,2$. Then the elements $a_1$ and $a_2$ generate a free group on two generators.
\end{lemma}

The key to obtain the elements $a_i$ and the sets $A_i$, $B_i$ in the context of Theorem \ref{Tits} is to use that there exist big contractions in the group, which send a large part of the space into a small part. Let us give some details. If $A\subset K$ and $\epsilon >0$ we denote $$A^{\epsilon}= \left\{ x \in K \ | \ d(x,A) < \epsilon \right\}.$$

Most of the article is dedicated to prove the following proposition.

\begin{proposition} \label{globalcontrcorollary}
	Let $K$ be a compact subset of $\R$. Let $G$ be a subgroup of $\mathrm{Homeo}(K)$ whose elements are locally monotonic on $K$, and whose action on $K$ does not have an invariant probability measure. Then there exists an integer $p$ such that the following statement holds: for any $\ep>0$, there exist finite subsets $A$ and $B$ of $K$ of cardinal $p$ and an element $g$ of $G$ such that $g(K\backslash A^\ep)\subset B^\ep$.
\end{proposition}

In what follows, we explain how to deduce Theorem \ref{Tits} from Proposition \ref{globalcontrcorollary}. To achieve this, we will combine the above property with the following general fact. 

\begin{proposition} \label{displacement}
	Let $G$ be a group acting on a set $X$ with no finite orbit. Then, for any finite subsets $A$, $B$ of $X$, there exists an element $g$ of $G$ such that 
	$$g(A) \cap B= \emptyset.$$ 
\end{proposition}
This result is quite elementary. We give a proof below, but we also have been communicated by Nicolàs Matte Bon that it follows from a general lemma of group theory (see Lemma 4.1 of \cite{Neu}).
\begin{proof}
	We prove the statement by induction on $p=|A|$. The case $p=1$ is immediate by using that the orbits of the group action are infinite. Let $p\geq 1$ and let us assume that the property is true for every subset $A$ of cardinal $p$ (and every finite subset $B$). Let $A$ be a set of cardinal $p+1$, and $B$ be a finite set. Let $a\in A$ and $A'=A\backslash \{a\}$, of cardinal $p$. By using the induction assumption, we define successively for $k\in\N$ an element $g_k\in G$ such that $g_k(A')$ is disjoint from $\bigcup_{j<k}g_j(A')$. Thus the sets $g_k(A')$ are pairwise disjoint. Two cases can occur:
	\begin{enumerate}
		\item The set $\{g_k(a),k\in\N\}$ is infinite: so we can find an infinite set of integers $E$ such that $g_k(a)\not\in B$ for every $k\in E$. Since the sets $g_k(A')$ are pairwise disjoint, we can find $k\in E$ such that $g_k(A')\cap B=\emptyset$. Then $g_k(A)\cap B=\emptyset$.
		\item The set $\{g_k(a),k\in\N\}$ is finite: so we can find an infinite set of integers $E$ and an element $b\in X$ such that $g_k(a)=b$ for every $k\in E$. Since the orbit of $b$ is infinite, we can find $h\in G$ such that $h(b)\not\in B$. Since the sets $g_k(A')$ are pairwise disjoint, we can find $k\in E$ such that $g_k(A')\cap h^{-1}(B)=\emptyset$. Then $h\circ g_k(A)\cap B=\emptyset$.
	\end{enumerate}
In any case we have found $g\in G$ such that $g(A)\cap B=\emptyset$, so the induction is done.	
\end{proof}

We now explain how to use the two above propositions to obtain Theorem \ref{Tits}. First, notice that in the context of the theorem, if $G$ does not have any invariant probability measure, we can apply Proposition \ref{displacement} because any finite orbit $F$ for the group action gives an invariant probability measure $\mu=\frac{1}{|F|}\sum_{x\in F}\delta_x$. So from the two above propositions, we obtain the following properties.

\begin{enumerate}
	\item (contractions) There exists $p$ such that for any $\ep>0$, there exist finite sets $A$, $B$ of cardinal $p$ and $g\in G$ such that $g(K\backslash A^\ep)\subset B^\ep$.
	\item (displacement) For any finite sets $A$, $B$ of $K$ there exists $g\in G$ such that $g(A)\cap B\not=\emptyset$.
\end{enumerate}

By the following proposition, these two properties are enough to satisfy the hypothesis of the Ping-Pong Lemma. Hence the group $G$ contains free groups of rank $2$.

\begin{proposition}
Let $G$ be a group acting continuously on a metric compact space $(K,d)$ so that the two properties (contractions) and (displacement) are satisfied. Then $G$ contains a free subgroup of rank $2$. 
\end{proposition}

\begin{proof}	
	Let us check first that in the first property, we can find $A$ and $B$ which do not depend on $\ep$. It means that for any respective neighborhoods $V_1$, $V_2$ of $A$, $B$, there exists $g\in G$ such that $g(K \setminus V_1)\subset V_2$. We say in this case that $(A,B)$ satisifies the contraction property.\\ \\
	To prove the existence of these sets, we apply the first property to a sequence $\ep_n$ which tends to $0$ we get a sequence $(g_n)$ in $G$ and finite sets $A_n$, $B_n$ of cardinal $p$ such that $g_n(K\backslash A_n^{\ep_n})\subset B_n^{\ep_n}$. Let $A_n=\{a_1^{(n)},\ldots,a_p^{(n)}\}$ and $B_n=\{b_1^{(n)},\ldots,b_p^{(n)}\}$. Taking subsequences if necessary, we can suppose that all the sequences $(a_1^{(n)}),\ldots,(a_p^{(n)}),(b_1^{(n)}),\ldots,(b_p^{(n)})$ converge with respective limits $a_1,\ldots,a_p,b_1,\ldots b_p$. Denoting $A=\{a_1,\ldots,a_p\}$ and $B=\{b_1,\ldots,b_p\}$, for any $\ep>0$, there exists $n$ such that $A_n^{\ep_n}\subset A^\ep$ and $B_n^{\ep_n}\subset B^\ep$ and then $g=g_n$ satisfies $g(K \setminus A^\ep)\subset B^\ep$.\\
	
	A second ingredient is the following remark: if $(A,B)$ satisfies the contraction property, and if $u,v$ are any elements of $G$, then $(u(A),v(B))$ also  satisfies the contraction property. Indeed,	
	if $V_1$ and $V_2$ are respective neighborhoods of $u(A)$ and $v(B)$ then  $u^{-1}(V_1)$ and $v^{-1}(V_2)$ are respective neighborhood of $A$ and $B$, so there exists $g\in G$ such that $g(K\backslash u^{-1}(V_1))\subset  v^{-1}(V_2)$  and so $h(K\backslash V_1)\subset V_2$ for $h=v\circ g\circ u^{-1}$.\\
	
	Now, take $(A,B)$ satisfying the contraction property. By the displacement property  there exists $u\in G$ such that $B\cap u(A)=\emptyset$. Then the pair $(A_1,B_1)=(u(A),B)$ also satisfies the contraction property and $A_1\cap B_1=\emptyset$. Using the displacement property a second time, we obtain $v\in G$ such that $(A_1\cup B_1)\cap v(A_1\cup B_1)=\emptyset$. The pair $(A_2,B_2)=(v(A_1),v(B_1))$ satisfies the contraction property and the sets $A_1$, $B_1$, $A_2$ and $B_2$ are pairwise disjoint.\\ 
	
	Finally, let $\ep>0$ be small enough so that the neighborhoods $A_1^\ep$,  $B_1^\ep$, $A_2^\ep$ and $B_2^\ep$ are still pairwise disjoint. Then the contraction property gives $a_1$ and $a_2$ such that $a_i(K\backslash A_i^\ep)\subset B_i^\ep$ for $i=1,2$, and the Ping-Pong Lemma applies: $a_1$ and $a_2$ generate a free group of rank $2$.
\end{proof}

The rest of the article is mostly devoted to the proof of Proposition \ref{globalcontrcorollary}. We will use random walks on $G$ to prove this proposition. With this probabilistic approach, we actually obtain a stronger result (see Proposition \ref{globalcontr}). We do not know if it is possible to obtain Proposition \ref{globalcontrcorollary} with deterministic methods.

\subsection{Outline of the proof}

Given a finitely generated group $G$ of locally monotonic homeomorphisms such that $G$ does not preserve any measure on $K$, we want to prove that $G$ contains ``big'' contractions in the sense of Proposition \ref{globalcontrcorollary}. We will find these contractions by probabilistic methods. We consider a random walk $g_n=f_{n-1}\circ\cdots\circ f_0$ on $G$ and we prove that, asymptotically, elements of the random walk have a high probability to contract a large part of $K$. The proof is divided into several steps.
\begin{enumerate}
\item We first obtain a local contraction phenomenon: In \cite{Mal}, it is proven in the case of a random walk on a group of homeomorphisms of the circle that, if there is no invariant measure then every point has almost surely a neighborhood which is contracted exponentially fast. By following closely the proof we will see that we can adapt this proof with minor modifications and that the phenomenon still holds here. The main point is that a local study allows to avoid the problematic points of $K$ where the monotonicity of a generator breaks.\\ \\
\item The globalization of the contraction phenomenon is harder in the case of a Cantor set than in the case of the circle, because of these ``break points'' (notion which will be defined properly). In the case of the circle, if an element $g$ brings two points $x$ and $y$ close to each other then it necessarily contracts a whole arc between $x$ and $y$. In our case it is not necessarily true that the interval $[x,y]$ is contracted, unless there is no break point between $x$ and $y$. Moreover, if $g$ is a large product of generators, then $g$ has a priori a lot of break points and its behavior might be wild. Nevertheless, in our case, we manage to analyse the break points of $g_n$: they are among the points $f_0^{-1}\cdots f_k^{-1}(c)$ where $c$ is a break point of a generator. It naturally leads to the study of $f_0^{-1}\cdots f_k^{-1}$: this defines an inverse random walk, and as such, one may expect some properties of convergence. And, indeed, using the local contraction property, we manage to prove that theses points $f_0^{-1}\cdots f_k^{-1}(c)$, and so the break points of $g_n$, accumulate on a finite set.\\ \\
\item From the two above steps, we obtain that away from a finite set, there is only a finite number of break points of $g_n$ to handle. It allows to prove that, as expected, almost surely, when $n$ is large, $g_n$ contracts a neighborhood of every point of $K$ except a finite number of exceptions. This allows us to deduce the existence of contractions in the sense of Proposition \ref{globalcontrcorollary}.
\end{enumerate}

\subsection*{Acknowledgement}

Both authors were supported by the ANR project Gromeov ANR-19-CE40-0007. 
We would like to thank Michele Triestino for his numerous comments on a previous version of the article.

\section{Preliminaries}\label{pre}

In this section we introduce notation and tools which are needed to implement the proof. 

Thanks to Lemma it is enough to prove Proposition \ref{globalcontrcorollary} for groups which are finitely generated. In all the sequel:\\
\begin{itemize}
	\item $K$ is a compact subset of $\R$;
	\item $G$ is a finitely subgroup of $\mbox{Homeo}(K)$ whose elements are locally monotonic;
	\item $S$ is a finite set of generators of $G$ as a group (it does not need to generate $G$ as a semigroup for our purpose, though it is obviously the case if we choose $S$ symmetric).
\end{itemize}

\subsection{Local monotonicity and break points}
Let us introduce some definitions and properties about local monotonicity.

\begin{definition}
	Let $g \in \mathrm{Homeo}(K)$. We say that a pair $\{a,b\}\subset K$ is a \emph{break pair} of $g$  if $a$ and $b$ are the two endpoints of a connected component of $\mathbb{R} \setminus K$ while $g(a)$ and $g(b)$ are not. We call \emph{break point} of $g$ any point of a break pair of $g$, and we denote by $\mathrm{break}(g)$ the set of break points of $g$.
\end{definition}

\begin{remark}
	If $G$ is a group of generalized interval exchange transformations, that we semiconjugate to a group $H$ of homeomorphisms of a Cantor set as  we did in the proof of Theorem \ref{TitsIET}, then break pairs of elements of $H$ corrrespond to discontinuities of elements of $G$. Conversely, if $G$ is a group of homeomorphisms of a compact set $K\subset \R$, in the quotient space obtained by collapsing the gaps of $K$, break pairs become discontinuity points. 
\end{remark}

The following fact follows easily from the compactness of $K$.
\begin{proposition}\label{finite}
Any locally monotonic $g\in \mathrm{Homeo}(K)$ has a finite number of break points.
\end{proposition}

Let also state another elementary proposition wich allows to define another notion useful for our purpose.

For any two real numbers $a$ and $b$, we denote by $[a,b]$ the set of real numbers between $a$ and $b$, even in the case where $b<a$.

\begin{proposition}\label{regularsegment} 
	Let $g\in \mathrm{Homeo}(K)$. If a segment $[a,b]$ does not contain a break pair of $g$, then $g$ induces a monotonic homeomorphism from $[a,b]\cap K$ onto $[g(a),g(b)]\cap K$. In particular, $g([a,b]\cap K)=[g(a),g(b)]\cap K$. In this case, we say that $g$ is \emph{regular} on $[a,b]$.
\end{proposition}

\begin{remark}
From the last two propositions we can deduce an interesting characterization of locally monotonic homeomorphisms of $K$.\\
$g:K\rightarrow K$ is a locally monotonic homeomorphism of $K$ if and only if there exist two finite coverings of $K$ by disjoints intervals $\{I_1,\ldots,I_p\}$ and $\{J_1,\ldots, J_p\}$ and for each $k=1,\ldots,p$ a homeomorphism $g_k$ from $I_k$ onto $J_k$ such that $g=g_k$ on $K\cap I_k$.\\
This characterization enlightens the similiraty between locally monotonic homeomorphisms and generalized interval exchange transformations. It also makes clear that the inverse of a locally monotonic homeomorphism of $K$ is still locally monotonic, which does not follow immediately from the definition.
\end{remark}

Let us conclude with the following consequence of the fact that $G$ is finitely generated and that its elements have a finite number of break points: 
\begin{lemma}\label{monotonic}
	There exists $r_0>0$ such that any homeomorphism in the generating set $S$ is regular on any interval of length less than $r_0$.
\end{lemma}
This property will be fundamental to study the action of $G$ at a local level: roughly speaking, given a path $g_n=f_{n-1}\circ\ldots \circ f_0$ ($f_k\in S$) in $G$ and a small interval $I$, as long as $g_n(I)$ does not grow too much, the action on $I$ is very similar to an action by homeomorphisms of the circle on a small arc.

\subsection{Random walk}\label{randomwalk}
 In order to use probabilistic methods, we fix a probability measure $P$ on $S$ with total support, and we consider the generated random walk. Let us introduce some notation.

Let $\Omega=S^{\mathbb{N}}$ where  $\mathbb{N}$ is the set of nonnegative integers. For any element $\omega$ in $\Omega$ and any $i \in \mathbb{N}$, we denote by $f_{\omega_{i}}$ the corresponding element of $S$ in the $i$-th factor of $S^{\mathbb{N}}$. We endow $\Omega$ with $\mathbb{P}=P^{\otimes\N}$ the product probability measure on $\Omega$ induced by $P$.

We say that a probability measure $\mu$ on $K$ is an \emph{invariant measure} for the action of $G$ on $K$ if, for any Borel set $A \subset K$ and any element $g$ of $G$, $\mu(g^{-1}(A))= \mu(A)$. We say that a probability measure $\mu$ on $K$ is a \emph{harmonic measure} (or \emph{stationary measure}) for the action of $(G,P)$ on $K$ if, for any Borel set $A \subset K$,
$$ \mu(A)= \sum _{s \in S} \mu(s^{-1}(A))P(\left\{ s \right\}).$$
By the Krylov-Bogoliubov trick the action of $G$ on $K$ admits a harmonic measure.
Such a measure $\mu$ on $K$ will be sometimes seen as a probability measure on $\mathbb{R}$ supported on $K$.

Denote by $\tau$ the one-sided shift on $\Omega$ and by $T$ the map
$$ \begin{array}{rrcl}
T: & \Omega \times K & \rightarrow & \Omega \times K \\
& (\omega,x) & \mapsto & (\tau(\omega),f_{\omega_{0}}(x))
\end{array}
.$$
For any integer $n$ in $\mathbb{N}$ and any point $x$ in $K$, we denote by $f_{\omega}^{n}(x)$ the point of $K$ which is the projection on $K$ of $T^{n}(\omega,x)$. Thus $f_{\omega}^n$ is the random composition $f_{\omega_{n-1}}\circ \cdots \circ f_{\omega_0}$. A probability measure $\mu$ is harmonic for the action of $(G,P)$ on $K$ if and only if the measure $\mathbb{P} \otimes \mu$ is invariant under $T$. Finally, we say that a harmonic measure $\mu$ is ergodic for the action of $G$ on $K$ if the measure $\mathbb{P} \otimes \mu$ is ergodic for the map $T$.  


\section{Local contractions}
The goal of the section is to prove the following proposition.
	
\begin{proposition}\label{localcontr}
	Suppose the action of $G$ on $K$ has no invariant probability measure. Then there exists $\lambda>0$ such that, for any point $x$ of $K$, for $\mathbb{P}$-almost every $\omega$ in $\Omega$, there exists a neighborhood $B$ of $x$ such that, for any $n\geq 0$, $f_\omega^n(B)$ has diameter less than $e^{-n\lambda}$. In particular,
	$$\forall n \geq 0,\forall y\in B, d(f_\omega^n(x),f_\omega^n(y)) \leq e^{-\lambda n}.$$
\end{proposition}		

\begin{remark}~	
	\begin{enumerate}
		\item By reducing $B$ we can get that for every $n\ge 0$, $f_\omega^n(B)$ has diameter less than $r_0$, where $r_0$ is given by Lemma \ref{monotonic}. So in the conlusion of the proposition we can also state that for every $n\ge 0$, $f_\omega^n$ is monotonic on $B$.
		\item If $L\subset K$ is a minimal closed invariant subset and if $x\in L$ then it is standard that for $\mathbb{P}$-almost every $\omega$ in $\Omega$, $(f_\omega^n(x))$ is dense in $L$. Combining with Proposition \ref{localcontr}, it implies that there exists a neighborhood $B$ of $x$ and an element $g=f_\omega^n\in G$ monotonic on $B$ such that $g(B)$ is a (strict) subset of $B$, which implies in particular that $g$ has a fixed point in $B$. In consequence, $G$ does not act freely on $K$.
		\item Working a little more one can obtain two elements $g,h$ such that $g(B)$ and $h(B)$ are two disjoint subsets of $B$, which implies that $g$ and $h$ generate a free semigroup. However, Proposition \ref{localcontr} is not sufficient to construct a free group, we need to understand a more global behavior of the random walk, what will be done in sections afterwards.
	\end{enumerate}
\end{remark}

To prove the proposition, we follow \cite{Mal} where a similar statement is proved for random walks on $\mathrm{Homeo}(\mathbb{S}^1)$, with only minor modifications.\\

Let us start with two elementary facts on harmonic measures:

\begin{lemma} \label{acatoms}~
\begin{enumerate}
\item	If $\mu$ is a harmonic measure then for any element $s$ in $S$, the measure $s_{*} \mu$ is absolutely continuous with respect to $\mu$. If $S$ is symmetric, we actually have that $g_*\mu$ is equivalent to $\mu$ for any $g$ in $G$.
\item	If a harmonic measure $\mu$ has atoms, then the action of $G$ on $K$ has a finite orbit, hence an invariant probability measure.
\end{enumerate}
\end{lemma}

\begin{proof}

If $A$ is a Borel set of $K$ such that $\mu(A)=0$, by applying the definition of harmonic measure, we obtain that $\sum _{s \in S} \mu(s^{-1}(A))P(\left\{ s \right\})=0$ where $P(\{s\})>0$, so for any $s \in S$, $\mu(s^{-1}(A))=0$. Thus for $s$ in $S$, $s_*\mu$ is absolutely continuous with respect to $\mu$. If $S$ is symmetric then we also have that $(s^{-1})_*\mu$ is absolutely continuous with respect to $\mu$, so $s^*\mu$ and $\mu$ are equivalent, and since $S$ is a generating set for $G$, we actually have that $g_*\mu$ and $\mu$ are equivalent for every $g$ in $G$.\\

	Suppose now that $\mu$ has atoms and let $x\in K$ be such that $\mu(\left\{ x \right\})=M$ is maximal. Using the definition of harmonic measure and the maximality of $M$, we obtain that, for any element $s$ in $S$, $\mu( \left\{ s^{-1}(x) \right\})=M$. Hence the finite set $\left\{ x \in K \ | \ \mu( \left\{ x \right\})=M \right\}$ is invariant under the action of $G$.  
\end{proof}

Finally, we need the following general measure theory lemma (Proposition 5 of \cite{Led}).

\begin{lemma} \label{generaltheory}
	Let $\mu_{1}$ and $\mu_{2}$ be Borel probability measures on $\mathbb{R}$. For any point $x$ in $\mathbb{R}$, we denote by $J(x)= \frac{d\mu_{2}}{d \mu_{1}}(x)$ the Radon-Nykodym derivative of $\mu_{2}$ with respect to $\mu_{1}$ at the point $x$. Then, for $\mu_{1}$ almost every point $x$ in $\mathbb{R}$, we have
	$$ J(x)= \lim_{ I \in \mathcal{I}_{x} \ \left|I \right| \rightarrow 0} \frac{\mu_{2}(I)}{\mu_{1}(I)},$$
	where $\mathcal{I}_{x}$ is the set of open intervals which contain the point $x$.
	Moreover, if 
	$$q(x)= \sup \left\{ \frac{\mu_{2}(I)}{\mu_{1}(I)} \ | \ I \in \mathcal{I}_{x} \right\},$$
	then $\log^{+}(q)= \max(\log(q),0)$ is an element of $L^{1}(\mu_{1})$. 
\end{lemma}

Now, we fix a harmonic measure $\mu$. We assume that $\mu$ is not $G$-invariant and has no atoms: it is a consequence of the assumptions of Proposition \ref{localcontr} thanks to Lemma \ref{acatoms}. We set 

$$J(\omega,x)= \frac{d (f_{\omega_{0}}^{-1})_{*} \mu}{d\mu}(x).$$
By definition, this is equivalent to say that $d(f_{\omega_{0}}^{-1})_{*}\mu=J(\omega,x) d\mu+d\nu$ where $\nu\perp\mu$ (if $S$ is symmetric, we actually have $d(f_{\omega_{0}}^{-1})_{*}\mu=J(\omega,x) d\mu$ by Lemma \ref{acatoms}).

Observe that
$$ \int_{\Omega \times K} J(\omega,x) d\mu(x) d \mathbb{P}(\omega)\leq \int_{\Omega} \int_{f_{\omega_{0}}(K)} d \mu d\mathbb{P}(\omega) =\int_{\Omega} \int_{K} d \mu d\mathbb{P}(\omega)=1.$$
Hence $J$ belongs to $L^{1}(\mu \otimes \mathbb{P})$.
Therefore the function $\log^{+}(J)$ also belongs to $L^{1}(\mu \otimes \mathbb{P})$ and the quantity
$$ h(\mu)= - \int_{\Omega \times K} \log(J(\omega,x)) d\mu(x) d\mathbb{P}(\omega)$$
is well-defined but its value can be $+\infty$.

\begin{lemma} \label{positiveentropy}
	$h(\mu)>0.$
\end{lemma}

\begin{proof}
	Convexity of the $-\log$ function and Jensen's inequality imply that
	$$h(\mu) \geq - \log\left(\int_{\Omega \times K} J(\omega,x) d\mu(x) d \mathbb{P}(\omega)\right)=0.$$
	Moreover, if $h(\mu)=0$ then we have an equality in Jensen's inequality, so, for $\mathbb{P} \otimes \mu$-almost every $(\omega,x) \in \Omega \times K$, $J(\omega,x)=1$. It implies that $f_{\omega_0 *}\mu=\mu$ for every $f_{\omega_0}$ in the generating set $S$, hence $\mu$ is $G$-invariant, which contradicts the assumption.
\end{proof}

Recall that any measure on $K$ can be seen as a measure on $\mathbb{R}$ supported on $K$. For $\epsilon >0$ and $(\omega,x) \in \Omega \times K$, let
$$J_{\epsilon}(\omega,x)= \sup \left\{ \frac{(f_{\omega_0}^{-1})_{*}\mu(I)}{\mu(I)} \ | \ I \in \mathcal{I}_{x}, \ |I| < \epsilon \right\}.$$
By Lemma \ref{generaltheory}, the function $\log^{+}(J_{\epsilon})$ belongs to $L^{1}(\mu)$. Hence the quantity
$$h_{\epsilon}(\mu)= - \int_{\Omega \times K} \log(J_{\epsilon}) d(\mu \otimes \mathbb{P})$$
is well-defined but its value can be $+\infty$. 

\begin{lemma} \label{approxentr}
	There exists $\epsilon_{0} >0$ such that, for any $0 <\epsilon< \epsilon_{0}$, $h_{\epsilon}(\mu) >0$.
\end{lemma}

\begin{proof}
	Observe that, for almost every point $(\omega,x)$ in $\Omega \times K$, for any sequence $(\epsilon_{n})_{n \geq 0}$ of positive real numbers which decreasingly converges to $0$, the sequence $(-\log(J_{\epsilon_{n}})(\omega,x))_{n \geq 0}$ is increasing and, by Lemma \ref{generaltheory},
	$$ \lim_{n \rightarrow + \infty} -\log(J_{\epsilon_{n}}(\omega,x))=- \log(J(\omega,x)).$$
	Hence, by the monotone convergence theorem $\displaystyle \lim_{n \rightarrow + \infty} h_{\epsilon_{n}}(\mu)=h(\mu)$ (the function $-\log(J_{\epsilon_{n}})$ is not necessarily positive but is bounded from below uniformly in $n$ by an integrable function so that the monotone convergence theorem still applies). Therefore $\displaystyle \lim_{\epsilon \rightarrow 0} h_{\epsilon}(\mu)=h(\mu)$ and Lemma \ref{approxentr} is a consequence of Lemma \ref{positiveentropy}.
\end{proof}

With all this setting we are ready to prove the main lemma:
\begin{lemma}\label{localcontrae}
	Suppose the action of $G$ on $K$ has no invariant probability measure. Let $\mu$ be a harmonic measure which is ergodic. Then there exists $\lambda>0$ such that for $\mathbb{P}\otimes\mu$-almost every $(\omega,x)$ in $\Omega\times K$, there exists a neighborhood $B$ of $x$ so that 
	$$\forall n \geq 0,\forall y\in B, d(f_\omega^n(x),f_\omega^n(y)) \leq  e^{-\lambda n}.$$
\end{lemma}
\begin{proof}
	Let $r_0 >0$ be given by Lemma \ref{monotonic}. We fix $\epsilon >0$ sufficiently small so that $h_{\epsilon}(\mu) >0$ and $\epsilon<r_0$. By ergodicity of the measure $\mu$ and the Birkhoff ergodic theorem, for $\mathbb{P} \otimes \mu$-almost every $(\omega,x)$ in $\Omega \times K$,
	\begin{equation}\label{Birkhoff1}
	\lim_{n \rightarrow + \infty} \frac{1}{n} \sum_{k=0}^{n-1} \log(J_{\epsilon}(T^{k}(\omega,x)))=- h_{\epsilon}(\mu).	\end{equation}
	
	For $x \in K$ we set
	$$Q(x)= \sup \left\{ \frac{|I|}{\mu(I)} \ | \ I \in \mathcal{I}_{x} \right\}.$$
	By Lemma \ref{generaltheory}, the function $\log^{+}(Q)$ belongs to $L^{1}(\mu)$. Hence as a consequence of the Birkhoff ergodic theorem applied to $\log^{+}(Q)$, for $\mathbb{P} \otimes \mu$-almost every $(\omega,x)$ in $\Omega \times K$,
		\begin{equation}\label{Birkhoff2}\lim_{n \rightarrow + \infty} \frac{\log^{+}(Q)(f_{\omega}^{n}(x))}{n}=0. \end{equation}
	
	Let us consider a point $(\omega,x)$ where both (\ref{Birkhoff1}) and (\ref{Birkhoff2}) hold. Then, fixing $0<\alpha<\beta<h_\ep(\mu)$, there exists $n_0\in\N$ such that 
	\begin{equation}\label{bound1}
	\forall n \geq n_0, \ \prod_{k=0}^{n-1} J_{\epsilon}(\tau^{k}(\omega),f_{\omega}^{k}(x)) \leq e^{- \beta n}
	\end{equation}
	 and
	\begin{equation}\label{bound2}
	\forall n \geq n_0, Q(f_\omega^n(x))\leq e^{\alpha n}
	\end{equation}
    
    	Let $y \in K$. For $n\geq 0$, let us denote $I_n=[f_\omega^n(x),f_\omega^n(y)]$ (recall the convention that $[a,b]$ is the set of real numbers between $a$ and $b$, even if $b<a$). We see that if $|I_n|\leq r_0$, then $f_{\omega_n}$ is regular on $I_n$ (in the sense of Proposition \ref{regularsegment}) so $f_{\omega_n}(I_n\cap K)=I_{n+1}\cap K$, and in particular $\mu(I_{n+1})=(f_{\omega_n}^{-1})_*\mu(I_n)$.
    
Take $0 <\ep < \min(r_0,1)$. For $n\geq n_0$, if we have $|I_k|\leq \ep$ for $k=0,\ldots, n-1$, then $\mu(I_{k+1})\leq J_\ep(\tau^k(\omega),f_{\omega_k}(x))\mu(I_k)$ for $k<n$, and so $\mu(I_n)\leq  e^{-\beta n}\mu(I_0)\leq e^{-\beta n}$ by (\ref{bound1}), and then $|I_n|\leq e^{(\alpha-\beta)n}$ by (\ref{bound2}). In particular if we have chosen $n_0$ large enough so that $e^{(\alpha-\beta)n}\leq \ep$, which is obviously possible since $\alpha<\beta$, it implies that $|I_n|\leq \ep$.
    
    So, if $y$ is close enough to $x$ so that $|I_{n_0}|\leq \ep$, we have by induction that $\forall n\geq n_0, |I_n|\leq e^{-\lambda n}$ where $\lambda=\beta-\alpha$. And if $y$ is close enough to $x$, the previous inequality also holds for $n\leq n_0$.

\end{proof}

\begin{lemma}\label{uniform}
	There exist a positive real number $\lambda_0$ and an open cover $\{B_k\}_{1\leq k\leq p}$ of $K$ such that for every $k=1,\ldots,p$, the set
	$$\left\{\omega\in \Omega,\forall n\geq 0, \forall y,z \in B_k, d(f_\omega^n(y),f_\omega^n(z))\leq e^{-n\lambda_0}\right\}$$
	has positive $\mathbb{P}$-probability.
\end{lemma}
\begin{proof}
	Let $x\in K$. The closure $\overline{\mathcal{O}(x)}$ of the orbit of $x$, is a $G$-invariant compact subset of $K$. Hence there exists an ergodic harmonic measure $\mu$ which is supported in $\overline{\mathcal{O}(x)}$. Lemma \ref{localcontrae} implies that there exist $x'\in \overline{\mathcal{O}(x)}$, a neighborhood $B$ and a positive number $\lambda$ such that the set
	$$A_1=\left\{\omega\in \Omega,\forall n\geq 0, \forall y,z \in B, d(f_\omega^n(y),f_\omega^n(z))\leq e^{-n\lambda}\right\}$$
	has $\mathbb P$-positive probability. Moreover, since $B\cap \overline{\mathcal{O}(x)}\not=\emptyset$, there exists a neighborhood $B_x$ of $x$ and an integer $m$ so that the set 
	$$A_2=\{\omega\in\Omega|f_\omega^m(B_x)\subset B\}$$
	has positive probability. Then the set $A=A_2 \cap \tau^{-m}(A_1)$ has also positive probability, and for $\omega\in A$ we have
	$$ \forall y,z \in B_x, \forall n\geq m, d(f_\omega^n(y),f_\omega^n(z))\leq e^{-(n-m)\lambda}.$$
	We deduce that if $\lambda'\in (0,\lambda)$, we can find a smaller neighborhood $B_x'$ of $x$  such that
	$$ \forall y,z \in B_x', \forall n\geq 0, d(f_\omega^n(y),f_\omega^n(z))\leq e^{-n\lambda'}.$$
	 We conclude by extracting a finite covering of $\{B_x'\}_{x\in K}$.
\end{proof}
We can finish the proof of Proposition \ref{localcontr}.
\begin{proof}[Proof of Proposition \ref{localcontr}]
	Let $\lambda_0$ be the number given by  Lemma \ref{uniform}. For each $x$ in $K$, let us denote by $A_x$ the set of events $\omega\in\Omega$ such that for every $\lambda<\lambda_0$, there exists a neighborhood $B$ of $x$ such that 
	$$\forall y\in B, \forall n\geq 0, d(f_\omega^n(x),f_\omega^n(y))\leq e^{-n\lambda}.$$
 By Lemma \ref{uniform}, we have $\mathbb{P}(A_x)\geq c$ for some $c>0$ which does not depend on $x$. Moreover, for any $k\geq 0$ one can check that $\tau^k(\omega)\in A_{f_\omega^k(x)}\Rightarrow \omega\in A_x$.  It implies the almost sure inequality $\mathbb{P}(A_x|\mathcal{F}_k)\geq \mathbb{P}(A_{f_\omega^n(x)})\geq c$, where $\mathcal{F}_k$ is the $\sigma$-algebra generated by the $k$ first projections of $\Omega$ (we used the convention  that $\mathbb{P}(A|\mathcal{F})$ is the conditional expectation of the indicatrix function of $A$ with respect to $\mathcal{F}$). Then, the $0-1$ Levy's law says that $\mathbb{P}(A_x|\mathcal{F}_k)$ converges almost surely to the caracteristic function of $A_x$ when $k$ tends to $+\infty$, so we conclude that $A_x$ has full probability.
\end{proof}

\section{Inverse random walk behaviour}
In this section we establish a result about the inverse random walk $f_{\omega_0}\circ\cdots \circ f_{\omega_{n-1}}$. It is actually a more general result, which holds for any random walk satisfying the contraction property proved in the previous section.

\subsection{Statement}
Let $(K,d)$ be a compact metric space, and let $P$ be a probability measure on $\mathrm{Homeo}(K)$. Let $\Omega=\mathrm{Homeo}(K)^\N$ endowed with $\mathbb{P}=P^{\otimes \N}$. For $\omega=(f_0,f_1,\ldots)$ in $\Omega$ we set $f_\omega^n=f_{n-1}\circ\cdots \circ f_0$ and $\bar{f}_{\omega}^n=f_0\circ\cdots\circ f_{n-1}$.

\begin{theorem}\label{inverse}
	We assume that the random walk $(f_\omega^n)_{n \in\N}$ satisfies the following `` fast local contractions property'': for any $x$ in $K$, for $\mathbb{P}$-almost any $\omega$ in $\Omega$, there exists a neighborhood $B$ of $x$ such that $\displaystyle \sup_{y\in B}\sum_{n=0}^{+\infty}d(f_\omega^n(x),f_\omega^n(y))<+\infty$.
	
	Then the inverse random walk $(\bar{f}_{\omega}^n)_{n\in\N}$ satisfies the following ``accumulation property'': there exists an integer $k$ such that for any $x$ in $K$, for $\mathbb{P}$-almost every $\omega$ in $\Omega$, the sequence $(\bar{f}_{\omega}^n(x))_{n\in\N}$ has at most $k$ cluster values.
\end{theorem}

\subsection{A probability statement}
We begin by proving a general statement about Markov chains that we will use.

\begin{proposition}\label{Markovsum}
	Let $(X_n)_{n\in\N}$ be a homogeneous Markov chain  defined on a probability space $(\Omega,\mathcal{F},\mathbb{P})$ and valued on a measurable state space $E$. Let $\phi:E\rightarrow \R_+$ be a measurable function, and let $C,\delta>0$. We assume that $\phi\leq C$ on $E$ and that for every starting point $x$ in $E$,
	$$\mathbb{P}\left(\sum_{n=0}^{+\infty}\phi(X_n)\leq C\Big|X_0=x\right)\geq \delta.$$
	Then $\displaystyle \sum_{n=0}^{+\infty}\phi(X_n)$ is almost surely finite, integrable and
	$$\E\left(\sum_{n=0}^{+\infty}\phi(X_n)\right)\leq \frac{C}{\delta}.$$
\end{proposition}

We will deduce the proposition from the following lemma:
\begin{lemma}\label{Markovlemma}
	Let $\hat{F}$ be a set of finite and infinite sequences in $E$ such that:
	\begin{itemize}
		\item every sequence of length $1$ belongs to $\hat{F}$;
		\item
		for any infinite sequence $u=(u_n)_{n\in\N}$ in $E^\N$, if every initial sequence $(u_0,u_1,\ldots,u_n)$ belongs to $\hat{F}$ then so does $u$.
	\end{itemize}     	
	We assume that there exists $\delta>0$ such that for every $x$ in $E$,
	$$\mathbb{P}\left((X_n)_{n\in\N}\in \hat{F}|X_0=x\right)\geq \delta.$$
	Then almost surely $(X_n)_{n\in\N}$ can be decomposed into a finite number $T$ of sequences $(X_0,\ldots,X_{N_1-1})$, $(X_{N_1},\ldots,X_{N_2-1})$,...,$(X_{N_{T-1}},\ldots)$, all of them in $\hat{F}$, where
	$T$ is a random variable in $\N$ such that $\mathbb{P}(T\geq k+1)\leq (1-\delta)^k$ for every integer $k$. 
\end{lemma}
\begin{proof}
	We define an increasing sequence of stopping times $(N_k)$ by induction. We set $N_0=0$ and then, if $N_k<+\infty$ we set $N_{k+1}=\inf\{N\geq N_k|(X_{N_k},\ldots, X_N)\not\in \hat{F}\}$ (so $N_{k+1}>N_k$ by the first property of $\hat{F}$). If $N_k=+\infty$ we stop the sequence. In this way, we obtain a decomposition of $(X_n)_{n\in\N}$ into a finite or infinite number of subsequences $(X_{N_{k-1}},\ldots, X_{N_k-1})$, all of them in $\hat{F}$. (by using the second property of $\hat{F}$ if $N_k=+\infty$).  Let $T$ be this number of subsequences.
	
	By assumption and by the strong Markov property,
	$$\mathbb{P}(N_{k+1}<+\infty|N_k<+\infty)=\mathbb{P}((X_{N_k+n})_{n\in\N}\not\in \hat{F}|N_k<+\infty)\leq 1-\delta,$$
	so $\mathbb{P}(N_k<+\infty)\leq (1-\delta)^k$, and so $\mathbb{P}(T\geq k+1)\leq (1-\delta)^k$. In particular, $T<+\infty$ a.s., and we get the result.
\end{proof}

\begin{proof}[Proof of Proposition \ref{Markovsum}]
	Let $\hat{F}$ be the set of sequences $(u_n)_{n\in I}$ such that $\displaystyle \sum_{n\in I}\phi(u_n)\leq C$. Then with the notations of the lemma, we can almost surely decompose $(X_n)_{n\in\N}$ in $T$ subsequences of the form $(X_n)_{n\in I_j}$ with $\displaystyle \sum_{n\in I_j}\phi(X_n)\leq C$, and so $\displaystyle \sum_{n=0}^{+\infty}\phi(X_n)\leq TC$. Then
	$$\E(T)=\sum_{k=0}^{+\infty}\mathbb{P}(T\geq k+1)\leq \sum_{k=0}^{+\infty}(1-\delta)^k=\frac{1}{\delta}$$
	so
	$$\E\left(\sum_{n=0}^{+\infty}\phi(X_n)\right)\leq \frac{C}{\delta}.$$
\end{proof}
\begin{remark}
	From the proof we actually obtain that $\displaystyle \sum_{n=0}^{+\infty}\phi(X_n)$ has an exponential moment since $\E[e^{\ep T}]<+\infty$ for $\ep$ small.
\end{remark}

\subsection{Proof of Theorem \ref{inverse}}
Let $G$ be the subsemigroup of $\mathrm{Homeo}(K)$  generated by the topological support of $\nu$. We need the following definition of ``$m$-proximality'' of $G$, generalizing the classical definition of proximality.

\begin{definition}
	Two points $x,y$ of $K$ are said to be \emph{proximal} for the action of $G$ if there exists a sequence $(g_n)_{n\in \N}$ in $G$ such that $d(g_n(x),g_n(y))\to 0$ as $n\to +\infty$.
	
	If $m\geq 2$ is an integer, the semigroup $G$ is said to be \emph{$m$-proximal} if any set $E\subset K$ of cardinal $m$ contains  at least two proximal distinct points.
	
	The smallest integer $m$ satisfying the above condition, if it exists, is called \emph{degree of proximality of $G$}.
\end{definition}

The classical definition of proximality of $G$ (every points $x$, $y$ are proximal) corresponds to the case $m=2$. The interest of this generalization is that, as we are going to see, under the assumption of the theorem the group $G$ is necessarily $m$-proximal for some $m$. Then we will prove that the conclusion of the theorem holds for $k=m-1$. In particular, if $G$ is proximal we obtain that for any $x$ in $K$ and $\mathbb{P}$-almost every $\omega$ in $\Omega$ the sequence $(\bar{f}_{\omega}^n(x))_{n\in\N}$ converges.

\begin{lemma}
	If the assumption of Theorem \ref{inverse} is satisfied, then there exists $m$ in $\N$ such that $G$ is $m$-proximal.
\end{lemma}

\begin{proof}
	As a consequence of the assumption, every point of $K$ has a neighborhood $B$ such that any points $x,y$ of $B$ are proximal. By compactness of $K$, we can cover $K$ by a finite number $k$ of such neighborhoods $B_1,\ldots, B_k$. Then by the pigeonhole principle, every set of $k+1$ points has at least two points in a same neighborhood $B_i$, so has at least two proximal points. So $G$ is $(k+1)$-proximal.
\end{proof}

We introduce the map $\Delta_m:K^m\rightarrow\R$ defined by $\Delta_m(x_1,\ldots,x_m)=\inf_{i\not= j}d(x_i,x_j)$.

\begin{lemma} \label{positiveprobacontraction} Under the assumption of Theorem \ref{inverse}, let $m$ be such that $G$ is $m$-proximal. Then there exist positive numbers $\delta$ and $C$ such that, for any points $x_1,\ldots,x_m$ of $K$, for a set of words $\omega$ of $\mathbb{P}$-probability larger than $\delta$, we have
	$$\sum_{n=0}^{+\infty}\Delta_m(f_\omega^n(x_1),\ldots,f_\omega^n(x_m))\leq C.$$	
\end{lemma}
\begin{proof}
	As a straightforward consequence of the assumption, for any point of $K$ there is a number $C>0$  and a neighborhood $B$ such that for a set of words $\omega$ of $\mathbb{P}$-probability larger than $\frac{1}{2}$ we have $\displaystyle \sum_{n=0}^{+\infty}d(f_\omega^n(x),f_\omega^n(y))\leq C$ for every $x,y$ in $B$. By compactness of $K$, we deduce that there exist $C>0$ and $\ep>0$ such that, for any $x$, $y$ in $K$ with $d(x,y)\leq \ep$, there is a set of words $\omega$ of $\mathbb{P}$-probability larger than $\frac{1}{2}$ such that $\displaystyle \sum_{n=0}^{+\infty}d(f_\omega^n(x),f_\omega^n(y))\leq C$.
	
	Since $G$ is $m$ proximal, for any $x_1,\ldots,x_m$ in $K$, there exist two indices $i, j$, an integer $k$ and a positive probability set of words $\omega$ such that $d(f_\omega^k(x_i),f_\omega^k(x_j))\leq \ep$. By compactness of $K$, there exist $\delta>0$ and an integer $k_0$ such that, for any points $x_1,\ldots,x_m$ of $K$, there exist a set of words $\omega$ of probability larger than $\delta$, two indices $i, j$ and an integer $k \leq k_0$. such that $d(f_\omega^k(x_i),f_\omega^k(x_j))\leq \ep$.
	
	Combining the two above arguments, for $x_1,...,x_m$ in $K$ there exist a set of words $\omega$ of probability larger than $\delta'=\frac{\delta}{2}$, two indices $i, j$ and an integer $k \leq k_0$ such that $\displaystyle \sum_{n=0}^{+\infty}d(f_\omega^{k+n}(x_i),f_\omega^{k+n}(x_j))\leq C$. Thus, $\displaystyle \sum_{n=0}^{+\infty}d(f_\omega^{n}(x_i),f_\omega^{n}(x_j))\leq C+ k \text{ diam}(K)$ and so for this set of words $\omega$,
	$$\sum_{n=0}^{+\infty}\Delta_m(f_\omega^n(x_1),\ldots,f_\omega^n(x_m))\leq C'$$
	where $C'=C+ k_0 \text{ diam}(K)$.
\end{proof}

We can deduce now a convergence result for the inverse random walk.
\begin{lemma}\label{inverselemma}
	Under the assumption of Theorem \ref{inverse}, let $m$ be such that $G$ is $m$-proximal. There exists a constant $C$ such that for every $x_1,\ldots,x_m$ in $K$,
	$$\E\left[\sum_{n=0}^{+\infty}\Delta_m(\bar{f}_\omega^n(x_1),\ldots,\bar{f}_\omega^n(x_m))\right]\leq C.$$
	In particular, the series $\sum_n \Delta_m(\bar{f}_\omega^n(x_1),\ldots,\bar{f}_\omega^n(x_m))$ converges for $\mathbb{P}$-almost every $\omega$.
\end{lemma}
\begin{proof}
We use the general probability result stated in Proposition \ref{Markovsum} with $E=K^m$, $X_n=(f_\omega^n(x_1),\ldots,f_\omega^n(x_m))$, $\phi=\Delta_m$. The assumption is satisfied by Lemma \ref{positiveprobacontraction}, and the conclusion gives a constant $C$ which does not depend on $x_1,\ldots,x_m$ such that
	$$\E\left[\sum_{n=0}^{+\infty}\Delta_m(f_\omega^n(x_1),\ldots,f_\omega^n(x_m))\right]\leq C.$$
Since $\omega \mapsto f_\omega^n$ and $\omega \mapsto \bar{f}_\omega^n$ have the same distribution, 
 $$\E \left[\Delta_m(\bar{f}_\omega^n(x_1),\ldots,\bar{f}_\omega^n(x_m)) \right]=\E \left[\Delta_m(f_\omega^n(x_1),\ldots,f_\omega^n(x_m))\right],$$ and the result follows.
\end{proof}

We are ready to prove the main result.
\begin{proof}[Proof of Theorem \ref{inverse}]
	Let $m$ be the degree of proximality of $G$. Let $E$ be a set of cardinal $m-1$ without proximal points. For any $\omega$ in $\Omega$, let us set $E_n(\omega)=\bar{f}_\omega^n(E)$. We are going to prove the following statements.
	\begin{enumerate}
		\item For any $x$ in $K$, for $\mathbb{P}$-almost every $\omega$ in $\Omega$, $d(\bar{f}_\omega^n(x),E_n(\omega))\to 0$ as $n\to +\infty$.
		\item For $\mathbb{P}$-almost every $\omega$ in $\Omega$, the sequence of sets $(E_n(\omega))_{n\in\N}$ converges for the Hausdorff metric on $K$.
	\end{enumerate}
	We recall that the Hausdorff metric is defined on the space of compact subsets of $K$ by $d_{H}(A,B)=\max(\sup_{x\in A}d(x,B),\sup_{y\in B}d(y,A))$.
	
	These statements obviously imply the wanted result since for any $x$ in $K$ and $\mathbb{P}$-almost every $\omega$ in $\Omega$, the set of cluster values of $(\bar{f}_\omega^n(x))_{n\in\N}$ is contained in the set $\displaystyle E_{\infty}(\omega)=\lim_{n\to+\infty} E_n(\omega)$, which has cardinal $m-1$.
	
	Let us begin by proving the first statement. Let us write $E=\{x_1,\ldots,x_{m-1}\}$, and let $x\in K$. Applying Lemma \ref{inverselemma} with $x_m=x$, we obtain that, for $\mathbb{P}$-almost every $\omega$,
	$$\sum_{n=0}^{+\infty}\Delta_m(\bar{f}_\omega^n(x_1),\ldots,\bar{f}_\omega^n(x_{m-1}),\bar{f}_\omega^n(x))<+\infty.$$
	Since $E$ does not contain any pair of proximal points, there exists $\ep_0>0$ such that  
	\begin{equation}\label{min}
	\forall g\in G,\forall x,y\in E, x\not=y \Rightarrow d(g(x),g(y))\geq \ep_0.  \end{equation}
	We deduce that except for a finite number of integers $n$ we have $$\Delta_m(\bar{f}_\omega^n(x_1),\ldots,\bar{f}_\omega^n(x_{m-1}),\bar{f}_\omega^n(x))=\inf_{1\leq i\leq m-1} d(\bar{f}_\omega^n(x),\bar{f}_\omega^n(x_i))=d(\bar{f}_\omega^n(x),E_n(\omega)),$$
	so that, for $\mathbb{P}$-almost every $\omega$,
	$$\sum_{n=0}^{+\infty}d(\bar{f}_\omega^n(x),E_n(\omega))<+\infty.$$
	In particular, $d(\bar{f}_\omega^n(x),E_n(\omega))\to 0$ as $n\to +\infty$.
	
	Let us prove the second statement. For $x$ in $K$, we apply this time Lemma \ref{inverselemma} with $x_m=f(x)$ where $f\in \mathrm{Homeo}(K)$ is integrated over $P$. Noting that $$\E[\Delta_m(\bar{f}_\omega^n(x_1),\ldots,\bar{f}_\omega^n(x_{m-1}),\bar{f}_\omega^n(f(x)))]
	=\E[\Delta_m(\bar{f}_\omega^n(x_1),\ldots,\bar{f}_\omega^n(x_{m-1}),\bar{f}_\omega^{n+1}(x))],$$ we deduce as above that, for $\mathbb{P}$-almost every $\omega$,
	$$\sum_{n=0}^{+\infty}\Delta_m(\bar{f}_\omega^n(x_1),\ldots,\bar{f}_\omega^n(x_{m-1}),\bar{f}_\omega^{n+1}(x))<+\infty,$$
	and then by using that $E$ does not contain any proximal point,
	$$\sum_{n=0}^{+\infty}d(\bar{f}_\omega^{n+1}(x),E_n(\omega))<+\infty.$$
	We fix $\omega$ such that this property holds for every $x$ in $E$. To simplify notation, we write $E_n$ instead of $E_n(\omega)$. Since $\displaystyle \sup_{x\in E} d(\bar{f}_\omega^{n+1}(x),E_n)=\sup_{x\in E_{n+1}} d(x,E_n)$ we obtain
	$$\sum_{n=0}^{+\infty}\sup_{x\in E_{n+1}} d(x,E_n)<+\infty.$$
	
	We write $d_H(E_n,E_{n+1})=\max(A_n,B_n)$ with 
	$$A_n=\sup_{x\in E_{n+1}} d(x,E_n) \mbox{ , } B_n=\sup_{x\in E_{n}} d(x,E_{n+1}).$$
	We know that $\displaystyle \sum_{n=0}^{+\infty}A_n<+\infty$. Moreover, by definition of $A_n$, for any integer $n$, there exists a map $\sigma:E_{n+1}\rightarrow E_n$ such that $d(x,\sigma(x))\leq A_n$ for every $x\in E_{n+1}$. Then, either $\sigma$ is bijective and $B_n=A_n=d_H(E_n,E_{n+1})$, or $\sigma$ is not and then there exist distinct points $x,y$ of $E_{n+1}$ such that $d(x,y)\leq 2 A_n$. In the latter case, $A_n\geq \frac{\ep_0}{2}$, where $\ep_0$ is defined by (\ref{min}). This second option only happens for a finite number of $n$ so finally
	$\sum_{n=0}^{+\infty}d_H(E_n,E_{n+1})<+\infty$. Thus $(E_n)_{n\in\N}$ is a Cauchy sequence and hence converges since $d_H$ is complete.
	
\end{proof}

\section{Global contractions}
In this section we come back in the context of Section \ref{pre}: $K$ is a compact subset of $\R$ and $G$ is a finitely generated subgroup of locally monotonic homeomorphisms of $K$. We prove a result about the behaviour of the sequence of homeomorphisms $(f_\omega^n)_{n\geq 0}$ for a typical event $\omega$. For random walks on the circle, almost surely there is a fixed finite number of repellor points, so that $(f_\omega^n)_{n\geq 0}$ contracts every arc which does not contains any of these repellors, see \cite{Mal}. We obtain an analogous statement here, though the global behaviour is not as easy to understand, and the proof is harder because of the lack of continuity. Precisely, the goal of this section is to prove the following result.

\begin{proposition}\label{globalcontr}
	Suppose that the action of group $G$ on $K$ does not admit any invariant probability measure. Then there exist positive numbers $\lambda$ and $C$ and an integer $p$ such that the following statement holds. For $\mathbb{P}$-almost every $\omega$ in $\Omega$, there exists a finite finite set $F$ of cardinal less than $p$ such that for any neighborhood $A$ of $F$, for any large enough $n$, the set $f_\omega^n(K\backslash A)$ is contained in a union of $p$ balls of radius $e^{-n\lambda}$.
\end{proposition}

Note that this proposition is stronger than Proposition \ref{globalcontrcorollary}, and so its proof will finish the proof of Theorem \ref{Tits} as well. To prove the proposition,  we are going to study the obstructions to the contracting behavior of $(f_\omega^n)_n$, and divide it into two types: the \textit{break points}, where the monotonicity of $f_\omega^n$ breaks, and the \textit{repellor points}, where $f_\omega^n$ sends small neighborhoods into a large set. We will give precise definitions later. As a tool to prove this proposition, we also establish the following dichotomy phenomenon: given two initial conditions and a typical event, the two corresponding trajectories either remain far away or synchronize exponentially fast. This dichotomy phenomenon is described in the next section. In the next two sections, we study the break points and the repellor points and prove that these two type of points mainly localize around a finite set. In a final section we combine all these arguments to prove Proposition \ref{globalcontr}.

\begin{remark}
Let us emphasize that the global behavior described by Proposition \ref{globalcontr} does not stricto sensu implies the local behavior described by Proposition \ref{localcontr}. It would be the case if we additionally state that for any $x$ in $K$, the random set $F$ given by Proposition  \ref{globalcontr} satisfies $\mathbb{P}(x\in F)=0$. Conversely, Proposition \ref{globalcontr} implies that we can chose $F$ such that this statement is true.
\end{remark}

\subsection{The dichotomy}

With technics similar to those that we used to deduce Proposition \ref{localcontr} from Lemma \ref{uniform}, we will prove the following lemma.

\begin{lemma} \label{synchro}
	Suppose the action of the group $G$ on $K$ has no invariant measure. Then there exist $\delta >0$ and $\lambda>0$ such that for any points $x$ and $y$ of $K$, for $\mathbb P$-almost every $\omega$ in $\Omega$, there exists a rank $n_0$ such that one of the two following assertions holds:
	\begin{enumerate}
		\item $\forall n\geq n_0, d(f^{n}_{\omega}(x),f^{n}_{\omega}(y)) \geq \delta$;
		\item  $\forall n\geq n_0, d(f^{n}_{\omega}(x),f^{n}_{\omega}(y))\leq e^{-n\lambda}$.
	\end{enumerate}
\end{lemma}
\begin{proof}
	 Fix $\lambda_0$ and a covering $\{B_k\}$ given by Lemma \ref{uniform}. Let $\delta>0$ so that any ball of radius $\delta$ is contained in a set $B_k$. Then, for $x,y$ in $K$, we set 
	$$\left\{\begin{array}{l}A_{x,y}=\{\omega\in\Omega,\forall \lambda<\lambda_0, \exists n_0\in \N, \forall n\geq n_0, d(f_\omega^n(x),f_\omega^n(y))\leq e^{-n\lambda}\},\\
	B_{x,y}=\{\omega\in \Omega, \exists n_0\in \N, \forall n\geq n_0, d(f_\omega^n(x),f_\omega^n(y))\geq \delta\},\\
	C_{x,y}=A_{x,y}\cup B_{x,y}.\end{array}\right.$$
	By Lemma \ref{uniform}, if $d(x,y)\leq \delta$ then $x,y$ both belong to some $B_k$. In this case, $\mathbb{P}(A_{x,y})\geq c$, where $c$ is a constant. Given $x,y$ in $K$, let us consider the stopping time 
	$$T=\inf\{n\geq 0| d(f_\omega^n(x),f_\omega^n(y))\leq \delta \}$$
	in $\N\cup\{+\infty\}$. On the one hand, for any $m$ in $\N$, we have 
	$$\mathbb{P}(\omega\in A_{x,y}|T(\omega)=m)\geq c$$ 
	so $\mathbb{P}(A_{x,y}|T<+\infty)\geq c$ and so $\mathbb{P}(C_{x,y}|T<+\infty)\geq c$. On the other hand, $\{T=+\infty\}\subset B_{x,y}$ so $\mathbb{P}(C_{x,y}|T=+\infty)=1\geq c$. thus we conclude that $\mathbb{P}(C_{x,y})\geq c$.
	
	Then, for any $n\geq 0$ we have that $\tau^n(\omega)\in C_{f_\omega^n(x),f_\omega^n(y)}\Rightarrow \omega\in C_{x,y}$,  so $\mathbb{P}(C_{x,y}|\mathcal{F}_n)\geq \mathbb{P}(C_{f_\omega^m(x),f_\omega^m(y)})\geq c$ where $\mathcal{F}_n$ is the $\sigma$-algebra generated by the $n$ first projections of $\Omega$. By the $0-1$ Levy's law, we conclude that $C_{x,y}$ has full probability.
		
\end{proof}

\subsection{Break points}

For a fixed $\omega \in \Omega$, we say that the sequence $(f_\omega^n)_{n\geq 0}$ is regular at $x$ if there exists an interval $I$ containing $x$ in its interior such that, for any $n\geq 0$, $f_\omega^n$ is regular on $I$ in the sense of Proposition \ref{regularsegment}.

\begin{lemma}\label{regular}
	There exists $m$ in $\N$ such that for $\mathbb P$-almost every $\omega$, there exist at most $m$ points where $(f_\omega^n)_{n\geq 0}$ is not regular.
\end{lemma}
\begin{proof}
Let us denote by $\Delta$ the union of the break points of all the elements of the generating set $S$ of $G$, which is finite. For $\omega\in \Omega$, let us set
$$\Delta_{\infty}(\omega)=\bigcup_{n\in\N} (f_\omega^n)^{-1}(\Delta).$$
 Using the elementary inclusion $\mathrm{break}(h \circ g)\subset \mathrm{break}(g)\cup g^{-1}(\mathrm{break}(h))$, valid for any $g,h\in \mathrm{Homeo}(K)$, we deduce that
$$ \mathrm{break}(f_\omega^n) \subset \bigcup_{k=1}^{n} (f_\omega^k)^{-1}(\mathrm{break}(f_{\omega_k})),$$
and so all the break points of $f_\omega^n$ for $n\geq 0$ belong to $\Delta_{\infty}(\omega)$. Let us set $F$ to be the set of all the cluster values of the sequences $((f_{\omega}^{n})^{-1}(x))_{n \geq 0}$ for $x\in\Delta$.

If $x\in K-F$, then there exists an interval $I$ containing $x$ in its interior which does not meet $\Delta_{\infty}(\omega)$ except maybe at $x$. Hence, for any $n\geq 0$, $I$ does not contain any break pair of $f_\omega^n$ and, by Proposition \ref{regularsegment}, $f_\omega^n$ is regular on $I$. Thus $(f_\omega^n)$ is regular at any point of $K-F$.

 Noting that $(f_{\omega}^{n})^{-1}=f_0^{-1}\circ\cdots\circ f_{n-1}^{-1}$, the sequence $((f_\omega^n)^{-1})_{n\in\N}$ can be seen as an inverse random walk so that Theorem \ref{inverse} implies. The random walk is generated by the inverse probability $P^{-1}$ on $S^{-1}$ (defined by $P^{-1}(\{s^{-1}\})=P(\{s\})$ for $s\in S$), where $S^{-1}$ still generates $G$ as a group and $P^{-1}$ has total support on $S^{-1}$. So Theorem \ref{localcontr} still applies and the (direct) random walk  $(f_{n-1}^{-1}\circ\cdots\circ f_{0}^{-1})_{n\in\N}$ satisifes the fast local contraction property, and so
 Theorem \ref{inverse} implies that there exists $m$ such that for any $x$ in $K$ and for $\mathbb{P}$ almost every $\omega$, the sequence $((f_\omega^n)^{-1}(x))_{n\in\N}$ has  at most $m$ cluster values. Applying this fact at each point of $\Delta$ we conclude that there exists $m'=m\times \mbox{Card}(\Delta)$ such that $F$ is finite with cardinal less than $m'$ for $\mathbb{P}$ almost every choice of $\omega$.

\end{proof}

\subsection{Repellors}

\begin{lemma}
	There exist positive numbers $\lambda$ and $\delta$ such for $\mathbb P$-almost every $\omega$ in $\Omega$ the following property holds.  For any point $x$ of $K$ where the sequence $(f_\omega^n)_{n\geq 0}$ is regular, one of the following holds.
	\begin{enumerate}
		\item Either there exist a neighborhood $B$ of $x$ and an integer $n_0$ such that $\forall n \geq n_0,  \mbox{diam}(f_\omega^n(B)) \leq e^{-\lambda n}.$
		\item Or for any neighborhood $B$ of $x$, there exists an integer $n_0$ such that $\forall n\geq n_0, \mbox{diam}(f_\omega^n(B))\geq \delta$.
	\end{enumerate}
We say that $x$ is an attractor of $(f_\omega^n)_{n\geq 0}$ if the first case holds and that $x$ is a repellor of $(f_\omega^n)_{n\geq 0}$ if the second case holds.
\end{lemma}

\begin{proof}
Fix $\lambda$ and $\delta$ given by Lemma \ref{synchro}. Let $D\subset K$ be a countable, dense subset of $K$ which contains the end points of $K$, that is the points which are either right or left isolated in $K$. For almost every $\omega$, the alternative stated in Lemma \ref{synchro} holds for every couple of points in $D$. Fix such a $\omega$, and let $x\in K$ be a point where the sequence $(f^n_\omega)_{n \geq 0}$ is regular. Let $I$ be an interval which contains $x$ in its interior such that $f_\omega^n$ is regular on $I$ for every $n$. Assume that the second case of the alternative of the statement does not hold for some neighborhood $B$ that we can assume contained in $I$. Then one can find $a\leq x$ in $B\cap D$, with strict inequality if $x$ is not isolated from the left, and in the same way one can also find $b\geq x$ in $B\cap D$, with strict inequality if $x$ is not isolated from the left. Then, since  we assumed that
	$$\liminf_{n\to +\infty} d(f_\omega^n(a),f_\omega^n(b)) <\liminf_{n\to +\infty} \mbox{diam}(f_\omega^n(B))<\delta,$$
	Lemma \ref{synchro} gives $n_0\in\N$ such that 
	$$\forall n\geq n_0, d(f_\omega^n(a),f_\omega^n(b))\leq e^{-\lambda n}.$$
	Then $B'=[a,b]\cap K$ is a neighborhood of $x$ such that, by monotonicity, 
	$$\forall n\geq n_0, \mbox{diam}(f_\omega^n(B'))\leq e^{-\lambda n}.$$
\end{proof}

\begin{lemma}\label{repellors}
There exists $m$ in $\N$, such that for almost every $\omega$, $(f_\omega^n)_{n\in\N}$ has at most $m$ repellors.
\end{lemma}
\begin{proof}
Assume that $x_1,\ldots, x_r$ are pairwise distinct regular points which are repellors of $(f_\omega^n)_{n\in\N}$. Let $I_1,..,I_r$ be pairwise disjoint interval containing $x_1,\ldots, x_r$ in $K$, where the sequence $(f_\omega^n)_n$ is regular. Then for each $k=1,\ldots,r$ there exist $a_k,b_k$ in $I_k$ such that for $n$ large, $d(f_\omega^n(a_k),f_\omega^n(b_k))\geq \delta$. Then for $n$ large,  $[f_\omega^n(a_1),f_\omega^n(b_1)],\ldots,[f_\omega^n(a_r),f_\omega^n(b_r)]$ are pairwise disjoint intervals of $K$, all of them with length larger than $\delta$. It implies that $r\delta\leq \mbox{diam}(K)$ and $r \leq \frac{\mbox{diam}(K)}{\delta}$.
\end{proof}

\subsection{Conclusion}

As a direct consequence of Lemmas \ref{regular} and \ref{repellors}, we have the following result.

\begin{proposition}\label{semilocalcontr}
	Suppose the action of $G$ on $K$ has no invariant probability measure. There exist $\lambda>0$ and $m\in \N$ such that for almost every $\omega$, for all $x$ in $K$ except at most $m$ exceptional points, there exists a neighborhood $B$ of $x$ such that for any large enough $n\geq 0$, $f_\omega^n(B)$ has diameter less than $e^{-n\lambda}$.
\end{proposition}
\begin{proof}	
	For $\omega\in \Omega$, we set $A=A_1\cup A_2$ where $A_1$ is the set of points of $K$ where $(f_\omega^n)$ is not regular in the sense of Lemma \ref{regular} and $A_2$ is the set of points where $(f_\omega^n)$ is not contracting in the sense of Lemma \ref{repellors}. As a consequence of these two lemmas, there exists an integer $p$ such that, almost surely, $A$ is finite of cardinal $p$. Every point outside $A$ is an attractor in the sense of Lemma \ref{repellors}, so it satisfies the conclusion of the statement by definition.
\end{proof}

\begin{remark}
	The statement of Proposition \ref{semilocalcontr} is very close to that of Proposition \ref{localcontr}, but somewhat more precise. If we denote $$\mbox{Exc}(\omega)=\{x\in K|\not\exists B \mbox{ neighborhood of } K| \exists n_0 \geq 0,\forall n\geq n_0, \mbox{diam}(f_\omega^n(B)) \leq e^{-\lambda n}\},$$ 
	then we obtain from Proposition \ref{localcontr} and Fubini Theorem that if $m$ is any measure on $K$, for almost every $\omega$,   $\mbox{Exc}(\omega)$ has null $m$-measure. Proposition \ref{semilocalcontr} states that it is finite. And though both statements seem to describe a local behavior, the statement of Proposition \ref{semilocalcontr} one has actually more global consequences. For instance, this statement (combined with one more use of Lemma \ref{synchro}) will allow us to prove Proposition \ref{globalcontr}.
\end{remark}

Let us prove Proposition \ref{globalcontr}.

\begin{proof}[Proof of Proposition \ref{globalcontr}]
	Let $\lambda$ be the number given by Proposition \ref{semilocalcontr}. Let us fix a set $D$ countable and dense in $K$. Let $\omega$ be an event such that the conclusion of Proposition \ref{semilocalcontr} holds and so that the alternative of Lemma \ref{synchro} holds for every couple of points in $D$. Let $F$ be the finite set given by Proposition \ref{semilocalcontr} for this $\omega$. Let $A$ be an open neighborhood of $F$, and $L=K\backslash A$.
	
	By the conclusion of Proposition \ref{semilocalcontr} and the compactness of $L$, we can cover $L$ by a finite number of balls $B_1,\ldots,B_r$ such that for any sufficiently large $n$, all the sets $f_\omega^n(B_i)$ have diameter less than $e^{-n\lambda}$. Each ball $B_i$ contains an element $x_i$ of $D$. So, there exists $n_0\in\N$ such that for any $i,j$ in $\{1,\ldots,r\}$, one (and only one) of the following assertions holds:
	\begin{enumerate}
		\item $\forall n\geq n_0, d(f_\omega^n(x_i),f_\omega^n(x_j))\leq e^{-n\lambda}$,
		\item $\forall n\geq n_0, d(f_\omega^n(x_i),f_\omega^n(x_j))\geq \delta$.
	\end{enumerate}
	Let $y_i=f_\omega^{n_0}(x_i)$. By compactness of $K$, we can divide the set of indexes $\{1,\ldots,r\}$ in $m$ subsets $J_1,\ldots,J_m$ such that each set $\{y_i,i\in J_k\}$ has diameter less than $\delta$, where $m$ is a constant depending only on $K$ and $\delta$. Then for $i,j\in J_k$ we have $d(f_\omega^n(x_i),f_\omega^n(x_j))\leq e^{-n\lambda}$ for $n\geq n_0$. Setting $L_k=\bigcup_{i\in J_k} B_i$, we have that for $n$ large enough, $f_\omega^n(L_k)$ has diameter less than $re^{-n\lambda}$. We deduce that for $n$ large enough, every set $f_\omega^n(L_k)$ is included in a ball of radius $e^{-n\lambda/2}$, and  $f_\omega^n(L)$ is contained in a union of $m$ balls of radius $e^{-n\lambda/2}$.
	
\end{proof}

\section{Complements}

In the case where we consider group of diffeomorphisms of a compact subset of $\R$, some results of the article can be made more precise. Namely, the contraction phenomenon can be stated in a smooth sense, and the contracting elements given by Proposition \ref{globalcontrcorollary} are actually Morse-Smale diffeomorphisms. In this section, we state precisely and prove such results. Throughout the whole section, $K$ denotes a compact subset of $\R$ and $G$ denotes a subgroup of $\mbox{Diff}^1(K)$.

\subsection{Smooth contractions}

\begin{proposition}\label{globalcontrdiffeo}
	Suppose that $G$ is finitely generated has no invariant probability measure. Defining a random walk on $G$ as in Section \ref{randomwalk}, there exist positive numbers $\lambda$ and $C$ and an integer $p$ such that the following statement holds. For $\mathbb{P}$-almost every $\omega$ in $\Omega$, there exists a finite set $F$ of cardinal less than $p$ such that, for any neighborhood $V$ of $F$, there exists a rank $n_0$ such that 
	$$\forall n\geq n_0, \sup_{x\in K\backslash V} |(f_\omega^n)'(x)|\leq e^{-n\lambda}.$$
\end{proposition}

\begin{proof}
	We use the contraction phenomenon given by Proposition \ref{semilocalcontr} combined with distortion arguments. For $g\in G$ and $B\subset K$ we define the (exponential) distortion
	$$D(g,B)=\frac{\sup_{x,y \in B} \frac{d(g(x),g(y))}{d(x,y)}}{\inf_{x,y \in B} \frac{d(g(x),g(y))}{d(x,y)}}.$$ 
Observe that, for $g,h \in G$ and for $B \subset K$, $D(g\circ h,B) \leq D(g,h(B))D(h,B)$.

	Let $\lambda$ be given by Proposition \ref{semilocalcontr}. We fix $\delta>0$ such that any generator of $G$ has on any ball of $K$ of diameter $\leq \delta$ a distortion less than $e^{\frac{\lambda}{2}}$. we fix also $\omega$ such that the conclusion of Proposition \ref{semilocalcontr} holds, and let $F$ be the correponding finite set of exceptional points. 
	 	At any point of $K\backslash F$, we can find a ball $B$ containing this point so that, for $n$ large enough,
	 	$$ \mbox{diam}(f_\omega^n(B))\leq e^{-n\lambda}.$$
	 	Shrinking $B$ if necessary, we can assume that the above inequality holds for any $n \geq 0$ and that
	 	$$\forall n\geq 0, \mbox{diam}(f_\omega^n(B))\leq \delta.$$
	 	From this last point and our choice of $\delta$ we can deduce that for any $n\geq 0$, $D(f_\omega^n,B)\leq e^{n\frac{\lambda}{2}}$. We deduce that, for any $n\geq 0$,
$$\sup_{x\in B} |(f_\omega^n)'(x)|\leq D(f_\omega^n,B) \frac{\mbox{diam}(f_\omega^n(B))}{\mbox{diam}(B)}\leq \frac{1}{\mbox{diam}(B)} e^{-n\frac{\lambda}{2}}.$$	 
If $V$ is a neighborhood of $A$ we deduce by compactness that for any large enough $n$,
	 	$$\sup_{x\in K\backslash V} |(f_\omega^n)'(x)|\leq e^{-n\frac{\lambda}{4}}.$$
\end{proof}
\begin{remark}\label{diffremark}~\\
\begin{enumerate}
\item In fact we the proof gives for $n$ large the estimate $\sup_{x,y\in K\backslash V} \frac{d(f_\omega^n(x),f_\omega^n(y))}{d(x,y)}\leq e^{-n\lambda}$, which is apriori more precise since the mean value theorem does not hold in $K$.
\item We could as well state a smooth version of the local version Proposition \ref{localcontr} by using the same argument. In its simplest form, it states that there exists $\lambda<0$ for any $x$ in $K$, for $\mathbb{P}$ almost every $\omega$, $\limsup_{n\to +\infty} \frac{\log |(f_\omega^n)'(x)|}{n}\leq \lambda$.
\end{enumerate}
\end{remark}
\begin{corollary}\label{smoothcontract}
	Suppose that $G$ has no invariant measure. Then there exists an integer $p$ such that for any $\ep>0$, there exist subsets $A$ and $B$ of $K$ of cardinal $p$ and a diffeomorphism $g$ in $G$ such that $|g'|\leq \ep$ on $K\backslash A^\ep$ and $|(g^{-1})'|\leq \ep$ on $K\backslash B^\ep$.
\end{corollary}

\begin{proof}
        We can assume that $G$ is finitely generated thanks to Lemma \ref{reduction}. We define a random walk on $G$ as in Section \ref{randomwalk} generated by a symmetric probability measure.  By proposition \ref{globalcontrdiffeo}, the probability that there exists a set $A$ of cardinal $p$ such that $|(f_\omega^n)'|\leq e^{-n\lambda}$ on $K\backslash A^\ep$ tends to $1$ as $n$ tends to $+\infty$.\\
	Since the random walk is symmetric, the probability that there exists a set $B$ of cardinal $p$ such that $|((f_\omega^n)^{-1})'|\leq e^{-n\lambda}$ on $K\backslash B^\ep$ tends also to $1$ as $n$ tends to $+\infty$.\\
	Take $n$ large enough such that the two sets of events above intersect each other and the conclusion follows.	
\end{proof}

\subsection{Existence of a Morse-Smale dffeomorphism}
We prove here that a group of diffeomorphisms of a compact subset of $\R$ without invariant probability measure always contains a Morse-Smale element.

Let us define first the notion of Morse-Smale diffeomorphism in this context, which is a very natural generalization of the standard definition.

\begin{definition}
Let $K\subset \R$ be a compact subset and $f\in \mbox{Diff}^1(K)$.	
	\begin{itemize}
		\item We say that a periodic point $x$ of $f$ of order $n$ is hyperbolic if $|(f^n)'(x)|\not=1$. It is attracting if $|(f^n)'(x)|<1$ and repelling if $|(f^n)'(x)|>1$.
		\item We say that $f$ is a Morse-Smale diffeomorphism if all its periodic points are hyperbolic (it implies that $f$ has a finite number of periodic points), and every forward orbit or backward orbit of $f$ tends to a periodic orbit.		
	\end{itemize}
\end{definition}

We use the following criterion to prove that a diffeomorphism is Morse-Smale.


\begin{proposition}\label{MorseSmaleCriterium}
	
	Let $g\in \mbox{Diff}^1(K)$. If there exist two disjoint open subsets $A$ and $B$ of $K$ such that $|g'|<1$ on $K\setminus A$ and $|(g^{-1})'|<1$ on $K\setminus B$, then $g$ is a Morse-Smale diffeomorphism.
\end{proposition}

To prove this proposition, we need the following general lemma.
\begin{lemma}
	Let $(X,d)$ be a metric compact space and $f:X\to X$ be a locally contracting map, in the sense that there exists $\delta>0$ such that 
	$$\forall x,y\in K, d(x,y)\leq \delta \Rightarrow d(f(x),f(y))<d(x,y).$$
	Then $f$ has a finite number of periodic orbits and every forward orbit of $f$ tends to one of them.
\end{lemma}
\begin{proof}
 First we notice that for $x,y$ in $X$, if $d(x,y)\leq \delta$ then $d(f^n(x),f^n(y))$ decreases to $0$: it is obviously decreasing and if $(x',y')$ is a cluster value of $(f^n(x),f^n(y))$ then $d(x',y')=d(f(x'),f(y'))$ so $x'=y'$.
 Say that $x\sim y$ if $d(f^n(x),f^n(y))\to 0$. There is a finite number of equivalence class $X_1,\ldots,X_m$, which are open and closed in $X$. Observe that $x\sim y$ if and only if $f(x)\sim f(y)$ so that $f$ induces a permutation of the sets $X_1,\ldots,X_m$. Hence there exists an integer $N$ such that $f^N$ leaves each of these sets invariant. Then $f^N$ is strictly contracting on each compact set $X_i$, so every forward orbit of $f^N$ converges to a unique fix point on each $X_i$. The result follows. 
\end{proof}

\begin{proof}[Proof of Proposition  \ref{MorseSmaleCriterium}]
The assumptions imply that $|(g^{-1})'|>1$ on $g(K\backslash A)$ and $|(g^{-1})'|<1$ on $K\backslash B$, so $g(K\backslash A)\subset B$. Since $A\cap B=\emptyset$, we deduce that $F=K\backslash A$ is a closed set which is invariant under $g$, and there exists $q<1$ such that $|g'|<q$ on $F$. In particular $g$ is a local contraction on $F$. Hence the lemma applies: $g$ has a finite number of periodic points on $K\backslash A$, and the forward orbit of $g$ of any point of $K\backslash A$ tends to a periodic orbit. Moreover, all these periodic points are hyperbolic since $|g'|<q$. In the same way, by applying the lemma to $g^{-1}$ on $K\backslash B$, we obtain that $g$ has a finite number of periodic points on $K\backslash B$, all of them hyperbolic, and the backward orbit of $g$ of any point of $K\backslash B$ tends to a periodic orbit.
	
	So $g$ has a finite number of periodic points, all of them hyperbolic. Moreover, every non empty totally invariant closed set meets either $K\backslash A$ or $K\backslash B$ and so contains a periodic point. We deduce that every forward orbit or backward orbit of $f$ tends to a periodic orbit and $f$ is Morse-Smale.		
\end{proof}

\begin{proposition}\label{MorseSmaleExistence}
	Let $G$ be a subgroup of $\mbox{Diff}^1(K)$ without invariant probablity measure. Then there exists a Morse-Smale diffeomorphism in $G$.
\end{proposition}
\begin{proof}	
We deduce this statement from Corollary	\ref{smoothcontract} in the same way that we deduced Theorem \ref{Tits} from Proposition \ref{globalcontrcorollary}.

\begin{enumerate}
	\item By compactness, there exist a couple of finite sets $(A,B)$ satisfying the conclusion of Corollary \ref{smoothcontract} for every $\ep>0$.
	\item If $(A,B)$ satisfies the conclusion of Corollary \ref{smoothcontract}, so does $(u(A),v(B))$ for $u,v$ in $G$.
	\item We deduce from Proposition \ref{displacement} that we can find $(A,B)$ satisfying the conclusion with $A\cap B=\emptyset$.
	\item For any $\ep>0$, there exists $g$ in $G$ such that $g$, $A^\ep$ and $B^\ep$ satisfy the assumptions of Proposition \ref{MorseSmaleCriterium} and $g$ is a Morse-Smale diffeomorphism.
\end{enumerate}	

\begin{remark}
It is possible to shortcut a little bit the proof by noting that in Corollary \ref{smoothcontract} we actually have the stronger statement that $g$ (resp. $g^{-1}$) is $\ep$-Lipschitz on $K\backslash A^\ep$ (resp. $K\backslash B^\ep$) (see Remark \ref{diffremark}), which gives a global information from which the existence of Morse-Smale is easier to deduce.  But the criterion of existence given by  Proposition  \ref{MorseSmaleCriterium} seemed to us interesting enough to be noted.
\end{remark}

\end{proof}

\end{document}